\documentclass[oneside,english]{amsart}
\usepackage[T1]{fontenc}
\usepackage[latin9]{inputenc}
\usepackage{geometry}
\usepackage{amsthm}
\usepackage{amssymb}

\makeatletter
\numberwithin{equation}{section}
\numberwithin{figure}{section}
\theoremstyle{plain}
\newtheorem{thm}{\protect\theoremname}
  \theoremstyle{definition}
  \newtheorem{defn}[thm]{\protect\definitionname}
  \theoremstyle{plain}
  \newtheorem{lem}[thm]{\protect\lemmaname}
  \theoremstyle{remark}
  \newtheorem{rem}[thm]{\protect\remarkname}
  \theoremstyle{plain}
  \newtheorem{prop}[thm]{\protect\propositionname}

\usepackage[all]{xy}
\usepackage[backref=page,colorlinks, linktocpage, urlcolor=blue, citecolor = blue, linkcolor = blue]{hyperref}
\usepackage{graphicx}
\usepackage{tikz}
\usepackage{tikz-cd}

\makeatother

\usepackage{babel}
  \providecommand{\definitionname}{Definition}
  \providecommand{\lemmaname}{Lemma}
  \providecommand{\propositionname}{Proposition}
  \providecommand{\remarkname}{Remark}
\providecommand{\theoremname}{Theorem}

\begin{document}
\title[Deformations of $\mathbb{A}^1$-cylindrical varieties]{Deformations of $\mathbb{A}^1$-cylindrical varieties}

\author{Adrien Dubouloz}

\address{IMB UMR5584, CNRS, Univ. Bourgogne Franche-Comté, F-21000 Dijon,
France.}

\author{Takashi Kishimoto}

\email{adrien.dubouloz@u-bourgogne.fr}

\address{Department of Mathematics, Faculty of Science, Saitama University,
Saitama 338-8570, Japan}

\email{tkishimo@rimath.saitama-u.ac.jp}

\thanks{The second author was partially funded by Grant-in-Aid for Scientific
Research of JSPS No. 15K04805. The research was initiated during a
visit of the first author at the University of Saitama and continued
during a stay of the second author at the University of Burgundy as
a CNRS Research Fellow. The authors thank these institutions for their
generous support and the excellent working conditions offered. }


\begin{abstract}
An algebraic variety is called $\mathbb{A}^{1}$-cylindrical if it
contains an $\mathbb{A}^{1}$-cylinder, i.e. a Zariski open subset
of the form $Z\times\mathbb{A}^{1}$ for some algebraic variety $Z$.
We show that the generic fiber of a family $f:X\rightarrow S$ of
normal $\mathbb{A}^{1}$-cylindrical varieties becomes $\mathbb{A}^{1}$-cylindrical
after a finite extension of the base. This generalizes the main result
of \cite{DK2} which established this property for families of smooth
$\mathbb{A}^{1}$-cylindrical affine surfaces. Our second result is
a criterion for existence of an $\mathbb{A}^{1}$-cylinder in $X$
which we derive from a careful inspection of a relative Minimal Model
Program ran from a suitable smooth relative projective model of $X$
over $S$. 
\end{abstract}

\maketitle

\section*{Introduction}

An algebraic variety is called $\mathbb{A}^{1}$-cylindrical (or affine-ruled
or $\mathbb{A}^{1}$-ruled) if it contains an $\mathbb{A}^{1}$-cylinder,
i.e. a Zariski open subset of the form $Z\times\mathbb{A}^{1}$ for
some algebraic variety $Z$. Such $\mathbb{A}^{1}$-cylinders appear
naturally in many recent problems and questions related to the geometry
of algebraic varieties, both affine and projective \cite{KMc99,DK1,DK2,DK3,DK4,CPW16,CPW16-2,CPW17,KPZ13,KPZ14,KPZ14-2,PZ14,PZ15}.
Clearly, there are only two $\mathbb{A}^{1}$-cylindrical smooth complex
curves: the affine line $\mathbb{A}^{1}$ and the projective line
$\mathbb{P}^{1}$. As a consequence of classical classification results,
every smooth projective surface of negative Kodaira dimension is $\mathbb{A}^{1}$-cylindrical,
and the same holds true for smooth affine surfaces by a deep result
of Miyanishi-Sugie and Fujita \cite{MS80}. But it is still an open
problem to find a complete and effective characterization of which
complex surfaces, possibly singular, contain $\mathbb{A}^{1}$-cylinders
\cite{KMc99}. The situation in higher dimension is even more elusive,
some natural class of examples of $\mathbb{A}^{1}$-cylindrical varieties
are known, especially in relation with the study of additive group
actions on affine varieties, but for instance the question whether
every smooth rational projective variety is $\mathbb{A}^{1}$-cylindrical
is still totally open. 

A natural way to try to produce new $\mathbb{A}^{1}$-cylindrical
varieties from known ones is to consider algebraic families $f:X\rightarrow S$
of such varieties. One hopes that the fiberwise $\mathbb{A}^{1}$-cylinders
could arrange themselves continuously to form a global relative $\mathbb{A}^{1}$-cylinder
in the total space $X$, in the form of a cylinder $U\simeq Z\times\mathbb{A}^{1}$
in $X$ for some $S$-variety $Z$, whose restriction to a general
closed fiber of $f:X\rightarrow S$ is equal to the initially prescribed
$\mathbb{A}^{1}$-cylinder in it. For families of relative dimension
one, it is a classical fact \cite{KM78} that a smooth fibration $f:X\rightarrow S$
whose general closed fibers are isomorphic to $\mathbb{A}^{1}$ indeed
restricts to trivial $\mathbb{A}^{1}$-bundle $Z\times\mathbb{A}^{1}$
over a dense open subset $Z$ of $S$. But on the other hand, the
existence of nontrivial conic bundles $f:X\rightarrow S$ shows that
it is in general too much to expect that fiberwise cylinders are restrictions
of global ones. Indeed, for such a nontrivial conic bundle, the general
closed fibers are isomorphic to $\mathbb{P}^{1}$, hence are $\mathbb{A}^{1}$-cylindrical,
but the generic fiber of $f:X\rightarrow S$ is a nontrivial form
of $\mathbb{P}^{1}$ over the function field $K$ of $S$: the latter
does not contains any open subset isomorphic to $\mathbb{A}_{K}^{1}$,
which prevents in turn the existence of a global $\mathbb{A}^{1}$-cylinder
in $X$ over an $S$-variety. Nevertheless, such an $\mathbb{A}^{1}$-cylinder
exists after extending the scalars to a suitable quadratic extension
of $K$, leading to the conclusion that the total space of any smooth
family $f:X\rightarrow S$ of $\mathbb{A}^{1}$-cylindrical varieties
of dimension one always contain a relative $\mathbb{A}^{1}$-cylinder,
possibly after an \'etale extension of the base $S$. 

A similar property is known to hold for certain families of relative
dimension $2$. More precisely, it was established in \cite[Theorem 3.8]{GMM14}
and \cite[Theorem7]{DK2} by different methods, involving respectively
the study of log-deformations of suitable relative projective models
and the geometry of smooth affine surfaces of negative Kodaira dimension
defined over non closed fields, that for smooth families $f:X\rightarrow S$
of complex $\mathbb{A}^{1}$-cylindrical affine surfaces, there exists
an \'etale morphism $T\rightarrow S$ such that $X_{T}=X\times_{S}T$
contains an $\mathbb{A}^{1}$-cylinder $U\simeq Z\times\mathbb{A}^{1}$
over a $T$-variety $Z$. The first main result of this article consists
of a generalization of this property to arbitrary families $f:X\rightarrow S$
of normal algebraic varieties defined over an uncountable base field,
namely: 
\begin{thm}
\label{thm:BaseChangeThm} Let $k$ be an uncountable field of characteristic
zero and let $f:X\rightarrow S$ be dominant morphism between geometrically
integral algebraic $k$-varieties. Suppose that for general closed
points $s\in S$, the fiber $X_{s}$ contains an $\mathbb{A}^{1}$-cylinder
$U_{s}\simeq Z_{s}\times\mathbb{A}^{1}$ over a $\kappa(s)$-variety
$Z_{s}$. Then there exists an \'etale morphism $T\rightarrow S$
such that $X_{T}=X\times_{S}T$ contains an $\mathbb{A}^{1}$-cylinder
$U\simeq Z\times\mathbb{A}^{1}$ over a $T$-variety $Z$. 
\end{thm}

We next turn to the problem of finding effective conditions on the
fiberwise $\mathbb{A}^{1}$-cylinders which ensure that a global relative
$\mathbb{A}^{1}$-cylinder exists, without having to take any base
change. The question is quite subtle already in the case of fibrations
of relative dimension $2$, as illustrated on the one hand by smooth
del Pezzo fibrations with non rational generic fiber, which therefore
cannot contain any $\mathbb{A}^{1}$-cylinder \cite{DK4}, and on
the other hand by examples of one parameter families $f:X\rightarrow S$
of smooth $\mathbb{A}^{1}$-cylindrical affine cubic surfaces whose
total spaces do not contain any $\mathbb{A}^{1}$-cylinder at all,
relative to $f:X\rightarrow S$ or not \cite{DK1}. Intuitively, a
global relative cylinder should exist as soon as the fiberwise $\mathbb{A}^{1}$-cylinders
are ``unique'', in the sense that the intersection of any two of
them is again an $\mathbb{A}^{1}$-cylinder. This holds for instance
for $\mathbb{A}^{1}$-cylinders inside non rational smooth affine
surfaces, and for families $f:X\rightarrow S$ of such surfaces, it
was indeed confirmed in \cite[Theorem 10]{DK2} that $X$ contains
a relative $\mathbb{A}^{1}$-cylinder $U\simeq Z\times\mathbb{A}^{1}$
over $S$, for which the rational projection $X\dashrightarrow Z$
coincides, up to birational equivalence, with the Maximally Rationally
Connected quotient of a relative smooth projective model $\overline{f}:\overline{X}\rightarrow S$
of $X$ over $S$. 

The natural generalization in higher dimension would be to consider
normal varieties $Y$ which contain $\mathbb{A}^{1}$-cylinders $U\simeq Z\times\mathbb{A}^{1}$
over non uniruled bases $Z$. But there is a second type of obstruction
for uniqueness, which does not appear in the affine case: the fact
that a given $\mathbb{A}^{1}$-cylinder $U\simeq Z\times\mathbb{A}^{1}$
in a variety $Y$ can actually be the restriction of a $\mathbb{P}^{1}$-cylinder
$\overline{U}\simeq Z\times\mathbb{P}^{1}$ inside $Y$, with the
effect that $Y$ then contains infinitely many distinct $\mathbb{A}^{1}$-cylinders
of the form $Z\times(\mathbb{P}^{1}\setminus\{p\})$, $p\in\mathbb{P}^{1}$,
all over the same base $Z$. This possibility is eliminated by restricting
the attention to varieties $Y$ containing $\mathbb{A}^{1}$-cylinders
$U\simeq Z\times\mathbb{A}^{1}$ for which the open immersion $Z\times\mathbb{A}^{1}\hookrightarrow Y$
cannot be extended to a birational map $Z\times\mathbb{P}^{1}\dashrightarrow Y$
defined over the generic point of $Z$. An $\mathbb{A}^{1}$-cylinder
with this property is called \emph{vertically maximal }in $Y$ (see
Definition \ref{def:VertMax}), and our second main result consists
of the following characterization: 
\begin{thm}
\label{thm:MainThm} Let $k$ be a field of characteristic zero and
let $f:X\rightarrow S$ be a dominant morphism between normal $k$-varieties
such that for general closed points $s\in S$, the fiber $X_{s}$
contains a vertically maximal $\mathbb{A}^{1}$-cylinder $U_{s}\simeq Z_{s}\times\mathbb{A}^{1}$
over a non uniruled $\kappa(s)$-variety $Z_{s}$. Then $X$ contains
an $\mathbb{A}^{1}$-cylinder $U\simeq Z\times\mathbb{A}^{1}$ for
some $S$-variety $Z$. 
\end{thm}

The article is organized as follows. The first section contains a
quick review of rationally connected and uniruled varieties and some
explanation concerning the minimal model program for varieties defined
over arbitrary fields of characteristic zero which plays a central
role in the proof of Theorem \ref{thm:MainThm}. In section two, we
establish basic properties of $\mathbb{A}^{1}$-cylindrical varieties.
Theorem \ref{thm:BaseChangeThm} is then derived in section three
from quite standard ``general-to-generic'' Lefschetz principle and
specialization arguments. Finally, section four is devoted to the
proof of Theorem \ref{thm:MainThm}, which proceeds through a careful
study of the output of a relative minimal model program applied to
a suitably constructed smooth projective model $\overline{f}:Y\rightarrow S$
of $X$ over $S$. 

\section{Preliminaries}

In what follows, unless otherwise stated, $k$ is a field of characteristic
zero, and all objects considered will be assumed to be defined over
$k$. A $k$-variety is a reduced scheme of finite type over $k$.
For a morphism $f:X\rightarrow S$ and another morphism $T\rightarrow S$,
the symbol $X_{T}$ will denote the fiber product $X\times_{S}T$.
In particular for a point $s\in S$, closed or not, we write $X_{s}=f^{-1}(s)=X\times_{S}\mathrm{Spec}(\kappa(s))$
where $\kappa(s)$ denotes the residue field of $s$. In addition,
if $T=\mathrm{Spec}(K)$ for a field $K$, then $X_{T}$ will also
sometimes be denoted by $X_{K}$. 

\subsection{Recollection on rational connectedness and uniruledness}
\begin{defn}
(See \cite[IV.3 Definition 3.2 and IV.1 Definition 1.1]{Ko96}) Let
$f:X\rightarrow S$ be an integral scheme over a scheme $S$. We say
that $X$ is: 

a) \emph{Rationally connected over} $S$ if there exists an $S$-scheme
$B$ and a morphism $u:B\times\mathbb{P}^{1}\rightarrow X$ of schemes
over $S$ such $u\times_{B}u:(B\times\mathbb{P}^{1})\times_{B}(B\times\mathbb{P}^{1})\rightarrow X\times_{S}X$
is dominant. 

b) \emph{Uniruled} over $S$ if there exists an $S$-scheme $B$ of
relative dimension $\dim(X/S)-1$ and a dominant rational map $u:B\times\mathbb{P}^{1}\dashrightarrow X$
of schemes over $S$. 

c) \emph{Ruled} over $S$ if there exists an $S$-scheme $B$ of relative
dimension $\dim(X/S)-1$ and a dominant birational map $u:B\times\mathbb{P}^{1}\dashrightarrow X$
of schemes over $S$. 
\end{defn}

A variety $X$ defined over a field $k$ is called rationally connected
(resp. uniruled, resp. ruled) if it is rationally connected (resp.
uniruled, resp. ruled) over $\mathrm{Spec}(k)$. Recall \cite[IV.3 3.2.5 and IV.1 Proposition 1.3]{Ko96}
that the first two notions are independent of the field over which
$X$ is defined. In particular, $X$ is rationally connected (resp.
uniruled) if and only if $X_{K}$ is rationally connected (resp. uniruled)
over $\mathrm{Spec}(K)$ for every field extension $k\subset K$.
In contrast, it is well-known that the property of being ruled depends
on the base field $k$: for instance a smooth conic $C\subset\mathbb{P}_{k}^{2}$
without $k$-rational point is uniruled but not ruled, but becomes
ruled after base extension to a suitable quadratic extension of $k$. 

The following lemma is probably well-known, but we include a proof
because of lack of an appropriate reference.
\begin{lem}
\label{lem:Factorization-Lemma} Let $Z$ be a non uniruled $k$-variety
and let $h:Y\rightarrow T$ be a surjective proper morphism between
normal $k$-varieties, with rationally connected general fibers. Then
every dominant rational map $p:Y\dashrightarrow Z$ factors through
a rational map $q:T\dashrightarrow Z$. 
\end{lem}

\begin{proof}
Since the property of being non uniruled is invariant under birational
equivalence, we can assume without loss of generality that $Z$ is
projective. Since $h:Y\rightarrow T$ is proper, for every blow-up
$\sigma:\tilde{Y}\rightarrow Y$ of $Y$, the composition $h\circ\sigma:\tilde{Y}\rightarrow T$
is again proper with rationally connected general fibers \cite[IV.3.3]{Ko96}.
So blowing-up $Y$ to resolve the indeterminacy of $p$, we can further
assume that $p$ is a morphism, and by shrinking $T$ that $h:Y\rightarrow T$
is faithfully flat and proper, with rationally connected fibers. Then
$Y$ is rationally chain connected over $T$. Let $u:\mathcal{C}=B\times\mathbb{P}^{1}\rightarrow Y$
be a morphism of algebraic varieties over $T$ witnessing this property.
The inverse image by $p\times p:Y\times_{T}Y\rightarrow Z\times Z$
of the diagonal is a closed subset $X$ of $Y\times_{T}Y$. If $X\neq Y\times_{T}Y$
then since $h\circ\mathrm{pr_{1}}:Y\times_{T}Y\rightarrow T$ is flat
hence open, the image of $Y\times_{T}Y\setminus X$ is a dense open
subset $T_{0}$ of $T$. Replacing $T$ by $T_{0}$, this would imply
that the image of $(p\circ u)\times_{B}(p\circ u):\mathcal{C}\times_{B}\mathcal{C}\rightarrow Z\times Z$
is not contained in the diagonal, in contradiction with the non-uniruledness
of $Z$. Thus $X=Y\times_{T}Y$ and so, $p$ is constant on the fibers
of $h:Y\rightarrow T$. By faithfully flat descent, there exists a
unique morphism $q:T\rightarrow Z$ such that $p=q\circ h$. 
\end{proof}
\begin{rem}
The conclusion of Lemma \ref{lem:Factorization-Lemma} does not hold
under the weaker assumption that the general fibers of $h:Y\rightarrow T$
are rationally chain connected. For instance, let $Y$ be the projective
cone over a smooth elliptic curve $Z\subset\mathbb{P}_{k}^{2}$ and
let $h:Y\rightarrow T=\mathrm{Spec}(k)$ be the canonical structure
morphism. Then $Y$ is rationally chain connected over $\mathrm{Spec}(k)$
and the projection $p:Y\dashrightarrow Z$ form the vertex of the
cone is a dominant rational map, which therefore does not factor through
$\mathrm{Spec}(k)$. 
\end{rem}

\subsection{\label{subsec:Minimal-Model-Program} Minimal Model Program over
non closed fields}

In the proof of Theorem \ref{thm:MainThm} given in section four below,
we will make use of minimal model program over arbitrary fields of
characteristic zero. We freely use the standard terminology and conventions
in this context, and just recall the mild adaptations needed to run
the minimal model program over a non closed field $k$ in a form appropriate
to our needs. It is well known (see e.g. \cite[$\S$ 2.2]{KM98} and
\cite[$\S$ 3.9]{K14}) that the minimal model program over an algebraically
closed field has natural equivariant generalizations to the case of
varieties with finite group actions, actually groups whose actions
on the Neron-Severi groups of the varieties at hand factor through
those of finite groups. This applies in particular to the situation
of a smooth projective morphism $\overline{f}:Y\rightarrow S$ between
normal quasi-projective varieties defined over a field $k$: after
the base change $\overline{f}_{\overline{k}}:Y_{\overline{k}}\rightarrow S_{\overline{k}}$
to an algebraic closure $\overline{k}$ of $k$, one can perform all
the basic steps of $K_{Y_{\overline{k}}}$-mmp with scalings relative
to $\overline{f}_{\overline{k}}:Y_{\overline{k}}\rightarrow S_{\overline{k}}$
as in \cite{BCHM} in an equivariant way with respect to the natural
action of the Galois group $G=\mathrm{Gal}(\overline{k}/k)$. Compared
to the genuine relative $K_{Y_{\overline{k}}}$-mmp with scalings,
this program runs in the category of varieties which are projective
over $S_{\overline{k}}$, with only terminal $G$-$\mathbb{Q}$-factorial
singularities, i.e. varieties with terminal singularities on which
every $G$-invariant Weil divisor is $\mathbb{Q}$-Cartier.

The termination of arbitrary sequences of $G$-equivariant flips is
not yet verified in a full generality, but as far as $K_{Y_{\overline{k}}}$
is not pseudo-effective over $S_{\overline{k}}$, it follows from
\cite[Corollary 1.3.3]{BCHM} that there exists a $G$-equivariant
$K_{Y_{\overline{k}}}$-mmp $\Theta:Y_{\overline{k}}\dashrightarrow\tilde{Y}$
over $S_{\overline{k}}$ with scalings by an $\overline{f}_{\overline{k}}$-ample
$G$-invariant divisor which ends with a $G$-Mori fiber space $\tilde{\rho}:\tilde{Y}\rightarrow\tilde{T}$
over a normal $S_{\overline{k}}$-variety $\tilde{T}$. That is, $\tilde{\rho}:\tilde{Y}\rightarrow\tilde{T}$
is a projective $G$-equivariant morphism between quasi-projective
$\overline{k}$-varieties with the following properties: $\tilde{\rho}_{*}\mathcal{O}_{\tilde{Y}}=\mathcal{O}_{\tilde{T}}$
, $\dim\tilde{T}<\dim\tilde{Y}$, $\tilde{Y}$ has only terminal $G$-$\mathbb{Q}$-factorial
singularities, the anti-canonical divisor $-K_{\tilde{Y}}$ is $\tilde{\rho}$-ample
and the relative $G$-invariant Picard number of $\tilde{Y}$ over
$\tilde{T}$ is equal to one. 

The birational map $\Theta:Y_{\overline{k}}\dashrightarrow\tilde{Y}$
is a composition of either divisorial contractions associated to successive
$G$-invariant extremal faces in the cone $\overline{\mathrm{NE}}(Y_{\overline{k}}/S_{\overline{k}})$
or flips which are all defined over $k$. The last morphism $\tilde{\rho}:\tilde{Y}\rightarrow\tilde{T}$
corresponds to a $G$-equivariant extremal contraction of fiber type
and is defined over $k$ as well. It follows that $\Theta:Y_{\overline{k}}\dashrightarrow\tilde{Y}$
and $\tilde{\rho}:\tilde{Y}\rightarrow\tilde{T}$ can be equivalently
seen as the base change to $\overline{k}$ of a sequence $\theta:Y\dashrightarrow Y'$
of $K_{Y}$-negative divisorial extremal contractions and flips between
$k$-varieties which are $\mathbb{Q}$-factorial over $k$ and projective
over $S$, and an extremal contraction of fiber type $\rho':Y'\rightarrow T$
between normal $k$-varieties, such that $-K_{Y'}$ is $\rho'$-ample
and the relative Picard number of $Y'$ over $T$ is equal to one. 

\section{$\mathbb{A}^{1}$-cylindrical varieties}

In this section, we introduce and establish basic properties of a
special class of ruled varieties called $\mathbb{A}^{1}$-cylindrical
varieties. 
\begin{defn}
Let $f:X\rightarrow S$ be a morphism of schemes. An \emph{$\mathbb{A}^{1}$-cylinder}
in $X$ over $S$ is a pair $(Z,\varphi)$ consisting of an $S$-scheme
$Z\rightarrow S$ and an open embedding $\varphi:Z\times\mathbb{A}^{1}\hookrightarrow X$
of $S$-schemes. We say that $X$ is\emph{ $\mathbb{A}^{1}$-cylindrical
over }$S$ if there exists an $\mathbb{A}^{1}$-cylinder $(Z,\varphi)$
in $X$ over $S$. 
\end{defn}

A variety $X$ defined over a field $k$ is called $\mathbb{A}^{1}$-cylindrical
over $k$ if it is $\mathbb{A}^{1}$-cylindrical over $\mathrm{Spec}(k)$.
Similarly as for ruledness, the property of being $\mathbb{A}^{1}$-cylindrical
depends on the base field $k$: a smooth conic $C\subset\mathbb{P}_{k}^{2}$
without $k$-rational point is not $\mathbb{A}^{1}$-cylindrical over
$k$ but becomes $\mathbb{A}^{1}$-cylindrical after base extension
to a suitable quadratic extension of $k$. 
\begin{defn}
A \emph{sub-$\mathbb{A}^{1}$-cylinder} of an $\mathbb{A}^{1}$-cylinder
$(Z,\varphi)$ in $X$ over $S$ is an $\mathbb{A}^{1}$-cylinder
$(Z',\varphi')$ in $X$ over $S$ for which there exists a commutative
diagram \[\begin{tikzcd}    Z' \times \mathbb{A}^1 \arrow[r,hook,"j"] \arrow[d,"\mathrm{pr}_{Z'}"] \arrow[rr,bend left=30,"\varphi'"] & Z\times \mathbb{A}^1 \arrow[r,"\varphi"] \arrow[d,"\mathrm{pr}_{Z}"] & X \\ Z' \arrow[r,hook,"i"] & Z \end{tikzcd}\]for
some open embeddings of $S$-schemes $i:Z'\hookrightarrow Z$ and
$j:Z'\times\mathbb{A}^{1}\hookrightarrow Z\times\mathbb{A}^{1}$.
Two $\mathbb{A}^{1}$-cylinders in $X$ over $S$ are called \emph{equivalent}
if they have a common sub-$\mathbb{A}^{1}$-cylinder over $S$. 
\end{defn}

\subsection{$\mathbb{A}^{1}$-cylinders and $\mathbb{P}^{1}$-fibrations}

Recall that a $\mathbb{P}^{1}$-fibration is a proper morphism $h:Y\rightarrow T$
between integral schemes whose fiber over the generic point of $T$
is a form of $\mathbb{P}^{1}$ over the function field $K$ of $T$.
Given an $\mathbb{A}^{1}$-cylinder $(Z,\varphi)$ in an algebraic
variety $X$ over $k$, the composition of $\varphi^{-1}:X\dashrightarrow Z\times\mathbb{A}^{1}$
with the projection $\mathrm{pr}_{Z}:Z\times\mathbb{A}^{1}\rightarrow Z$
extends on a suitable complete model $Y$ of $X$ to a $\mathbb{P}^{1}$-fibration
$h:Y\rightarrow T$ over a complete model $T$ of $Z$, restricting
to a trivial $\mathbb{P}^{1}$-bundle over a non empty open subset
$Z_{0}$ of $Z\subset T$. Conversely, the total space $Y$ of a $\mathbb{P}^{1}$-fibration
$h:Y\rightarrow T$ is $\mathbb{A}^{1}$-cylindrical over $T$ provided
that $h$ admits a rational section $H\subset Y$. Indeed, if so,
there exists a dense open subset $Z$ of $T$ such that $h^{-1}(Z)\simeq Z\times\mathbb{P}^{1}$
and $H\cap h^{-1}(Z)\simeq Z\times\{\infty\}$ for some fixed $k$-rational
point $\infty\in\mathbb{P}^{1}$, which implies in turn that the open
subset $h^{-1}(Z)\cap(Y\setminus H)$ of $Y$ is isomorphic to $Z\times\mathbb{A}^{1}$. 

The following characterization will be useful for the proof of Theorem
\ref{thm:MainThm}:
\begin{lem}
\label{lem:Factorization-MFS} Let $h:Y\rightarrow T$ be a surjective
proper morphism between normal varieties over a field $k$ of characteristic
zero, with irreducible and rationally connected general fibers. Suppose
that $Y$ contains an $\mathbb{A}^{1}$-cylinder $(Z,\varphi)$ for
some non-uniruled $k$-variety $Z$. Then $h:Y\rightarrow T$ is a
$\mathbb{P}^{1}$-fibration and there exists a sub-$\mathbb{A}^{1}$-cylinder
$(Z',\varphi')$ of $(Z,\varphi)$ and commutative diagram \[\begin{tikzcd}    Z' \times \mathbb{A}^1 \arrow[r,hook,"\varphi'"] \arrow[d,"\mathrm{pr}_{Z'}"] & Y \arrow[d,"h"] \\ Z' \arrow[r,hook,"i"] & T \end{tikzcd}\]
for some open embedding $i:Z'\hookrightarrow T$. In particular, $T$
is not uniruled. 
\end{lem}

\begin{proof}
By shrinking $Z$, we can assume that it is smooth and affine. Letting
$\overline{Z}$ be a smooth projective model of $Z$, the composition
$\mathrm{pr}_{Z}\circ\varphi^{-1}$ defines a dominant rational map
$\mathrm{pr}_{Z}\circ\varphi^{-1}:Y\dashrightarrow\overline{Z}$ which
lifts to a $\mathbb{P}^{1}$-fibration $p:\tilde{Y}\rightarrow\overline{Z}$
on some blow-up $\sigma:\tilde{Y}\rightarrow Y$ of $Y$. Since $\overline{Z}$
is not uniruled, and $h\circ\sigma:\tilde{Y}\rightarrow T$ is proper
with rationally connected general fibers, it follows from Lemma \ref{lem:Factorization-Lemma}
that $p$ factors through a dominant rational map $q:T\dashrightarrow\overline{Z}$.
So $\dim T\geq\dim\overline{Z}$ and since $\dim\overline{Z}=\dim\tilde{Y}-1\geq\dim T$,
we conclude that $\dim T=\dim Y-1$. This implies that $h\circ\sigma:\tilde{Y}\rightarrow T$
is a $\mathbb{P}^{1}$-fibration, hence that $h$ is a $\mathbb{P}^{1}$-fibration.
Since the general fiber of $p:\tilde{Y}\rightarrow\overline{Z}$ are
irreducible, $q$ has degree $1$, hence is birational. Now it suffices
to choose for $Z'$ an open subset of $Z\subset\overline{Z}$ on which
$q^{-1}$ restricts to an isomorphism onto its image. 
\end{proof}

\subsection{Birational modifications preserving $\mathbb{A}^{1}$-cylinders}

In contrast with ruledness, the property of containing an $\mathbb{A}^{1}$-cylinder
is obviously not invariant under birational equivalence. Nevertheless,
it is stable under certain particular birational modifications which
we record here for later use: 
\begin{lem}
\label{lem:Cylinder-birational-pullback} Let $(Y,\Delta)$ be a pair
consisting of a normal $k$-variety and a reduced divisor on it, let
$\theta:Y\dashrightarrow Y'$ be a birational map to a normal $k$-variety
$Y'$ and let $\Delta'=\theta_{*}(\Delta)$ be the proper transform
of $\Delta$ on $Y'$. Then the following hold:

a) If $\theta$ is an isomorphism in codimension one then $Y\setminus\Delta$
is $\mathbb{A}^{1}$-cylindrical over $k$ if and only if so is $Y'\setminus\Delta'$.

b) If $\theta$ is a proper morphism and $Y'\setminus\Delta'$ contains
an $\mathbb{A}^{1}$-cylinder $(Z',\varphi')$ over $k$ then there
exists a sub-cylinder $(Z,\varphi)$ of $(Z',\varphi')$ such that
$(Z,\theta^{-1}\circ\varphi)$ is an $\mathbb{A}^{1}$-cylinder in
$Y\setminus\Delta$ over $k$.

c) If $\theta$ is a proper morphism, each irreducible component of
pure codimension one of the exceptional locus $\mathrm{Exc}(\theta)$
of $\theta$ is uniruled and $Y\setminus\Delta$ contains an $\mathbb{A}^{1}$-cylinder
$(Z,\varphi)$ for some non-uniruled $k$-variety $Z$ then there
exists a sub-cylinder $(Z',\varphi')$ of $(Z,\varphi)$ such that
$(Z',\theta\circ\varphi')$ is a $\mathbb{A}^{1}$-cylinder in $Y'\setminus\Delta'$. 
\end{lem}

\begin{proof}
If $\theta$ is an isomorphism in codimension one, then it restricts
to an isomorphism $\theta:U\rightarrow U'$ between open subsets $U\subset Y$
and $U'\subset Y'$ whose complements $X$ and $X'$ have codimension
at least $2$ in $Y$ and $Y'$ respectively. Given a cylinder $(Z,\varphi)$
in $Y\setminus\Delta$, the inverse image by $\varphi$ of $X\cap(Y\setminus\Delta)$
has codimension at least two in $Z\times\mathbb{A}^{1}$, hence does
not dominate $Z$. Consequently, there exists a dense open subset
$Z'\subset Z$ such that $(Z',\varphi'=\varphi|_{Z\times\mathbb{A}^{1}})$
is a sub-$\mathbb{A}^{1}$-cylinder of $(Z,\varphi)$ whose image
is contained in $U\cap(Y\setminus\Delta)$, and the composition $(Z',\theta\circ\varphi')$
is then an $\mathbb{A}^{1}$-cylinder in $U'\cap Y'\setminus\Delta'\subset Y'\setminus\Delta'$.
Reversing the roles of $Y$ and $Y'$, this yields a). 

If $\theta$ is a proper morphism, then $C=\theta(\mathrm{Exc}(\theta))$
has codimension at least $2$ in $Y'$ because $Y'$ is normal. So
the restriction of $\mathrm{pr}_{Z'}$ to ${\varphi'}^{-1}(C)$ cannot
be dominant. This guarantees the existence of a dense open subset
$Z$ of $Z'$ such that $(Z,\varphi=\varphi'|_{Z\times\mathbb{A}^{1}})$
is a sub-$\mathbb{A}^{1}$-cylinder of $(Z',\varphi')$ whose image
is contained $Y'\setminus\Delta'\cup C$. Then $(Z,\theta^{-1}\circ\varphi)$
is an $\mathbb{A}^{1}$-cylinder in $Y\setminus\Delta\cup\mathrm{Exc}(\theta)\subset Y\setminus\Delta$,
which proves b).

Finally to prove c), we observe that since $Z$ is not uniruled, the
restriction of $\mathrm{pr}_{Z}$ to the inverse image by $\varphi$
of a uniruled irreducible component of pure codimension one of $\mathrm{Exc}(\theta)$
cannot be dominant. This implies that the restriction of $\mathrm{pr}_{Z}$
to $\varphi^{-1}(\mathrm{Exc}(\theta))$ is not dominant hence that
there exists a dense open subset $Z'\subset Z$ such that $\varphi(Z'\times\mathbb{A}^{1})\subset Y\setminus\Delta\cup\mathrm{Exc}(\theta)$.
Then $(Z',\varphi'=\varphi|_{Z'\times\mathbb{A}^{1}})$ is a sub-$\mathbb{A}^{1}$-cylinder
of $(Z,\varphi)$ with the desired property.
\end{proof}

\subsection{Uniqueness properties of $\mathbb{A}^{1}$-cylinders}

The fact that $\mathbb{P}_{k}^{1}$ contains infinitely many non-equivalent
cylinders $\mathbb{P}_{k}^{1}\setminus\{p\}$ over $k$, parametrized
by the set of its $k$-rational points $p\in\mathbb{P}_{k}^{1}(k)$,
shows that in general a given $k$-variety $X$ can contain many non
equivalent cylinders even when their respective base spaces are non-uniruled.
To ensure some uniqueness property of $\mathbb{A}^{1}$-cylinders,
we are led to introduce the following notion:
\begin{defn}
\label{def:VertMax}Let $f:X\rightarrow S$ be a morphism of schemes.
An $\mathbb{A}^{1}$-cylinder $(Z,\varphi)$ in $X$ over $S$ is
said to be \emph{vertically maximal} in $X$ over $S$ if for every
generic point $\xi$ of $Z$, the open embedding $\xi\times\mathbb{A}^{1}\hookrightarrow X$
induced by $\varphi$ cannot be extended to a morphism $\xi\times\mathbb{P}^{1}\rightarrow X$. 
\end{defn}

\noindent The next result can be thought as another geometric variant
of Iitaka and Fujita strong cancellation theorems \cite{IiFu77}. 
\begin{prop}
\label{prop:UniqueCylinder-vertMax}Let $X$ be $k$-variety containing
a vertically maximal $\mathbb{A}^{1}$-cylinder $(Z,\varphi)$ over
a non uniruled $k$-variety $Z$. Then every $\mathbb{A}^{1}$-cylinder
in $X$ over $k$ is equivalent to $(Z,\varphi)$. 
\end{prop}

\begin{proof}
Let $U=\varphi(Z\times\mathbb{A}^{1})$ be the open image of $Z\times\mathbb{A}^{1}$
in $X$. By shrinking $Z$ if necessary, we can assume that $Z$ is
affine and that all fibers of the projection $\mathrm{pr}_{Z}\circ\varphi^{-1}:U\rightarrow Z$
are closed in $X$. Let $(T,\psi)$ be another cylinder in $X$ with
image $V=\psi(T\times\mathbb{A}^{1})$, let $W=U\cap V$ and let $Z_{0}$
and $T_{0}$ be the open images of $W$ in $Z$ and $T$ by the morphisms
$\mathrm{pr}_{Z}\circ\varphi^{-1}$ and $\mathrm{pr}_{T}\circ\psi^{-1}$
respectively. Since $\mathrm{pr}_{T}\circ\psi^{-1}:W\rightarrow T_{0}$
is a surjective morphism with uniruled fibers and $Z$, whence $Z_{0}$,
is not uniruled, there exists a unique surjective morphism $\alpha:T_{0}\rightarrow Z_{0}$
such that $\mathrm{pr}_{Z}\circ\varphi^{-1}:W\rightarrow Z_{0}$ factors
as $\mathrm{pr}_{Z}\circ\varphi^{-1}=\alpha\circ(\mathrm{pr}_{T}\circ\psi^{-1}):W\rightarrow T_{0}\rightarrow Z_{0}$.
So for every point $t\in T_{0}$, there exists a unique $z=\alpha(t)\in Z_{0}$
such that $\psi(\mathrm{pr}_{T}^{-1}(t))\cap U$ is equal to $\varphi(\mathrm{pr}_{Z}^{-1}(z))\cap V$.
Since by hypothesis $\varphi(\mathrm{pr}_{Z}^{-1}(z))\simeq\mathbb{A}_{\kappa(z)}^{1}$
is closed in $X$, it follows that $\psi(\mathrm{pr}_{T}^{-1}(t))=\varphi(\mathrm{pr}_{Z}^{-1}(z))$.
This implies in turn that $\psi(T_{0}\times\mathbb{A}^{1})\subset\varphi(Z_{0}\times\mathbb{A}^{1})$
and that we have a commutative diagram: \[\begin{tikzcd} \psi(T_{0}\times\mathbb{A}^{1})   \arrow[r,hook] \arrow[d,swap,"\mathrm{pr}_{T}\circ\psi^{-1}"] & \varphi(Z_{0}\times\mathbb{A}^{1})  \arrow[d,"\mathrm{pr}_{Z}\circ\varphi^{-1}"] \arrow[r,hook] & U \arrow[d] \\ T_0 \arrow[r,"\alpha"] & Z_0 \arrow[r,hook] & Z. \end{tikzcd}\]It
follows in particular that $\alpha$ is also injective, hence an isomorphism.
Thus $(T_{0},\psi|_{T_{0}\times\mathbb{A}})$ is a sub-$\mathbb{A}^{1}$-cylinder
of $(Z,\varphi)$, which shows that $(Z,\varphi)$ and $(T,\psi)$
are equivalent. 
\end{proof}

\section{Proof of Theorem \ref{thm:BaseChangeThm}}

We now proceed to the proof of Theorem \ref{thm:BaseChangeThm}. By
hypothesis, $f:X\rightarrow S$ is a dominant morphism between geometrically
normal algebraic varieties defined over an uncountable field $k$
of characteristic zero, with the property that for general closed
points $s\in S$, the fiber $X_{s}$ contains a cylinder $(Z_{s},\varphi_{s})$
over a $\kappa(s)$-variety $Z_{s}$. Letting $X_{\eta}$ be the fiber
of $f$ over the generic point $\eta$ of $S$, the existence of an
\'etale morphism $T\rightarrow S$ such that $X\times_{S}T$ is $\mathbb{A}^{1}$-cylindrical
over $T$, is equivalent to that of a finite extension $L\subset L'$
of the function field $L$ of $S$ such that $X_{\eta}\times_{\mathrm{Spec}(L)}\mathrm{Spec}(L')$
is $\mathbb{A}^{1}$-cylindrical over $L'$. In fact, the following
specialization lemma implies that it is enough to find any extension
$L'$ of $L$ for which $X_{\eta}\times_{\mathrm{Spec}(L)}\mathrm{Spec}(L')$
is $\mathbb{A}^{1}$-cylindrical over $L'$:
\begin{lem}
\label{lem:CylinderafterFinite extension} Let $X$ be a variety defined
over a field $k$ of characteristic zero and let $k\subset K$ be
any field extension. If $X_{K}$ is $\mathbb{A}^{1}$-cylindrical
over $K$ then there exists a finite extension $k\subset k'$ such
that $X_{k'}$ is $\mathbb{A}^{1}$-cylindrical over $k'$. 
\end{lem}

\begin{proof}
By hypothesis, there exists an open embedding $\varphi:Z\times\mathbb{A}^{1}\hookrightarrow X_{K}$
for some $K$-variety $Z$. This open embedding is defined over a
finitely generated sub-extension $L$ of $K$, i.e. there exists an
open embedding $\varphi_{0}:Z_{0}\times\mathbb{A}^{1}\hookrightarrow X_{L}$
of $L$-varieties such that $\varphi$ is obtained from $\varphi_{0}$
by the base change $\mathrm{Spec}(K)\rightarrow\mathrm{Spec}(L)$.
Being finitely generated over $k$, $L$ is the function field of
an algebraic variety $S$ defined over $k$ and we can therefore view
$X_{L}$ as the fiber $\mathfrak{X}_{\eta}$ of the projection $\mathrm{pr}_{S}:\mathfrak{X}=X\times S\rightarrow S$
over the generic point $\eta$ of $S$. Let $\Delta$ and $T$ be
the respective closures of $\mathfrak{X}_{\eta}\setminus\varphi_{0}(Z_{0}\times\mathbb{A}^{1})$
and $\varphi_{0}(Z_{0}\times\{0\})$ in $\mathfrak{X}$. The projection
$\mathrm{pr}_{Z_{0}}:Z_{0}\times\mathbb{A}^{1}\rightarrow Z_{0}$
induces a rational map $\rho:\mathfrak{X}\setminus\Delta\dashrightarrow T$
whose generic fiber is isomorphic to $\mathbb{A}^{1}$ over the function
field of $T$. It follows that there exists an open subset $Y\subset T$
over which $\rho$ is regular and whose inverse image $V=\rho^{-1}(Y)$
is isomorphic to $Y\times\mathbb{A}^{1}$. Now for a general closed
point $s\in S$, the fiber $\mathfrak{X}_{s}$ of $\mathrm{pr}_{S}$
over $s$ is isomorphic to $X_{\kappa(s)}$, where $\kappa(s)$ denotes
the residue field of $s$, and contains an open subset $V_{s}$ isomorphic
to $Y_{s}\times\mathbb{A}^{1}$. The induced open immersion $Y_{s}\times\mathbb{A}^{1}\hookrightarrow X_{\kappa(s)}$
provides the desired $\mathbb{A}^{1}$-cylinder over the finite extension
$\kappa(s)$ of $k$. 
\end{proof}
Let $\overline{k}$ be an algebraic closure of $k$ and let $f_{\overline{k}}:X_{\overline{k}}\rightarrow S_{\overline{k}}$
be the morphism obtained by the base extension $\mathrm{Spec}(\overline{k})\rightarrow\mathrm{Spec}(k)$.
Since $S$ is geometrically integral, $S_{\overline{k}}$ is integral
and its field of functions $\overline{k}(S_{\overline{k}})$ is an
extension of the field of functions $L$ of $S$. If the generic fiber
of $f_{\overline{k}}$ becomes $\mathbb{A}^{1}$-cylindrical after
the base change to some extension of $\overline{k}(S_{\overline{k}})$
then by the previous lemma, the generic fiber $X_{\eta}$ of $f:X\rightarrow S$
becomes $\mathbb{A}^{1}$-cylindrical after the base change to a finite
extension of $L$. We can therefore assume from the very beginning
that $k=\overline{k}$ is an uncountable algebraically closed field
of characteristic zero. Up to shrinking $S$, we can further assume
without loss of generality that it is affine and that for every closed
point $s$ in $S$, $X_{s}$ contains a cylinder $(Z_{s},\varphi_{s})$
over a $k$-variety $Z_{s}$. Since $X$ and $S$ are $k$-varieties,
there exists a subfield $k_{0}\subset k$ of finite transcendence
degree over $\mathbb{Q}$ such that $f:X\rightarrow S$ is defined
over $k_{0}$, i.e. there exists a morphism of $k_{0}$-varieties
$f_{0}:X_{0}\rightarrow S_{0}$ and a commutative diagram \[\begin{tikzcd} X \arrow[d,swap,"f"] \arrow[r] & X_0 \arrow[d, "f_0"] \\ S \arrow[d] \arrow[r] & S_0 \arrow[d] \\   \mathrm{Spec}(k) \arrow[r] &  \mathrm{Spec}(k_{0})\end{tikzcd}\]in
which each square is cartesian. The field of functions $L_{0}=k_{0}(S_{0})$
of $S_{0}$ is an extension of $k_{0}$ of finite transcendence degree
over $\mathbb{Q}$, and since $k$ is uncountable and algebraically
closed, there exists a $k_{0}$-embedding $i:L_{0}\hookrightarrow k$
of $L_{0}$ in $k$. Letting $(X_{0})_{\eta_{0}}$ be the fiber of
$f_{0}$ over the generic point $\eta_{0}:\mathrm{Spec}(L_{0})\rightarrow S_{0}$
of $S_{0}$, the composition $\Gamma(S_{0},\mathcal{O}_{S_{0}})\hookrightarrow L_{0}\hookrightarrow k$
induces a $k$-homomorphism $\Gamma(S_{0},\mathcal{O}_{S_{0}})\otimes_{k_{0}}k\rightarrow k$
defining a closed point $s:\mathrm{Spec}(k)\rightarrow\mathrm{Spec}(\Gamma(S_{0},\mathcal{O}_{S_{0}})\otimes_{k_{0}}k)=S$
of $S$ for which we obtain the following commutative diagram 

\[\xymatrix@!=16pt{ & X_s \ar[dl] \ar@{->}'[d][dd] \ar[rr] & & X \ar[dl] \ar[dd]_(0.5){f} \\ (X_0)_{\eta_0} \ar[rr] \ar[dd] & & X_0 \ar[dd]_(0.4){f_0} \\ & \mathrm{Spec}(k) \ar[dl]_{i^*} \ar@{->}'[r][rr]^{s} & & S \ar[r] \ar[dl]  & \mathrm{Spec}(k) \ar[dl] \\ \mathrm{Spec}(L_0) \ar[rr]^{\eta_0} & & S_0 \ar[r]  & \mathrm{Spec}(k_0). &  }\]  Since
the bottom square of the cube above is cartesian by construction,
we have 
\[
(X_{0})_{\eta_{0}}\times_{\mathrm{Spec}(L_{0})}\mathrm{Spec}(k)\simeq X_{0}\times_{S_{0}}\mathrm{Spec}(k)\simeq X\times_{S}\mathrm{Spec}(k)=X_{s}.
\]
Since by hypothesis $X_{s}$ is $\mathbb{A}^{1}$-cylindrical over
$k$, we conclude that $(X_{0})_{\eta_{0}}\times_{\mathrm{Spec}(L_{0})}\mathrm{Spec}(k)$
is $\mathbb{A}^{1}$-cylindrical over $k$. Lemma \ref{lem:CylinderafterFinite extension}
then guarantees that there exists a finite extension $L_{0}\subset L_{0}'$
such that $(X_{0})_{\eta_{0}}\times_{\mathrm{Spec}(L_{0})}\mathrm{Spec}(L_{0}')$
is $\mathbb{A}^{1}$-cylindrical over $L_{0}'$. Finally, the tensor
product $L\otimes_{L_{0}}L_{0}'$ decomposes as a direct product of
finitely many finite extensions $L'$ of $L$ with the property that
$X_{\eta}\times_{\mathrm{Spec}(L)}\mathrm{Spec}(L')$ is $\mathbb{A}^{1}$-cylindrical
over $L'$, which completes the proof of Theorem \ref{thm:BaseChangeThm}. 
\begin{rem}
Combined with Proposition \ref{prop:UniqueCylinder-vertMax}, Lemma
\ref{lem:CylinderafterFinite extension} implies that for a $k$-variety
$X$, the property of containing a vertically maximal $\mathbb{A}^{1}$-cylinder
over a non uniruled variety is independent of the base field. Indeed,
by Lemma \ref{lem:CylinderafterFinite extension} if $X_{K}$ contains
a cylinder for some arbitrary field extension $k\subset K$, then
$X_{k'}$ contains a cylinder for a finite extension $k\subset k'$.
Letting $k''$ be the Galois closure of the extension $k\subset k'$
in an algebraic closure of $k'$, Proposition \ref{prop:UniqueCylinder-vertMax}
implies that the translates of a given cylinder $(Z,\varphi)$ in
$X_{k''}$ over $k''$ by the action of the Galois group $G=\mathrm{Gal}(k''/k)$
are all equivalent. Since $G$ is a finite group, it follows that
there exists a dense affine open subset $Z_{0}$ of $Z$, an action
of $G$ on $Z_{0}$ lifting to a $G$-action on $Z_{0}\times\mathbb{A}^{1}$
such that the induced open embedding $(Z_{0},\varphi|_{Z_{0}\times\mathbb{A}^{1}})\hookrightarrow X_{k''}$
is $G$-equivariant. The quotients $(Z_{0}\times\mathbb{A}^{1})/G$
and $Z_{0}/G$ are then affine varieties defined over $k$ while the
projection $Z_{0}\times\mathbb{A}^{1}\rightarrow Z_{0}$ and the open
embedding $\varphi|_{Z_{0}\times\mathbb{A}^{1}}:Z_{0}\times\mathbb{A}^{1}\hookrightarrow X_{k''}$
descend respectively to a locally trivial $\mathbb{A}^{1}$-bundle
$\pi:(Z_{0}\times\mathbb{A}^{1})/G\rightarrow Z_{0}/G$ and an open
embedding $\psi:(Z_{0}\times\mathbb{A}^{1})/G\hookrightarrow X_{k''}/G\simeq X$.
A cylinder in $X$ over $k$ is then obtained by restricting $\psi$
to the inverse image of a dense open subset of $Z_{0}/G$ over which
$\pi$ is a trivial $\mathbb{A}^{1}$-bundle. 
\end{rem}

\section{Proof of Theorem \ref{thm:MainThm}}

We first consider the case where $f:X\rightarrow S$ is a smooth projective
morphism whose general closed fibers contain vertically maximal $\mathbb{A}^{1}$-cylinders
over non uniruled varieties. The case of an arbitrary morphism $f:X\rightarrow S$
between normal algebraic varieties is then deduced by considering
a suitably constructed smooth relative projective model of $X$ over
$S$. 

\subsection{Case of a smooth projective morphism }
\begin{prop}
\label{prop:Relative-MMP} Let $\overline{f}:Y\rightarrow S$ be a
smooth projective morphism between normal $k$-varieties and let $\Delta\subset Y$
be a divisor on $Y$ such that for a general closed point $s\in S$,
$Y_{s}\setminus\Delta_{s}$ contains an $\mathbb{A}^{1}$-cylinder
$(Z_{s},\varphi_{s})$ over a non uniruled $\kappa(s)$-variety $Z_{s}$.
Then there exists a $K_{Y}$-MMP $\theta:Y\dashrightarrow Y'$ relative
to $\overline{f}:Y\rightarrow S$ whose output $\overline{f}':Y'\rightarrow S$
has the structure of a Mori conic bundle $\rho':Y'\rightarrow T$
over a non uniruled normal $S$-variety $h:T\rightarrow S$. Furthermore,
for a general closed point $s\in S$, there exists a sub-cylinder
$(Z_{s}',\varphi'_{s})$ of $(Z_{s},\varphi_{s})$ and a commutative
diagram \[\begin{tikzcd}[column sep=large] Z'_{s}\times\mathbb{A}^{1}  \arrow[d,swap,"\mathrm{pr}_{Z'_{s}}"] \arrow[r,"\theta_{s}\circ\varphi'_{s}"] & Y_s'\arrow[d, "\rho'_s"] \\ Z'_s  \arrow[r,hook, "\alpha_s" ] & T_s\end{tikzcd}\]where
the top and bottom arrows are open embeddings. 
\end{prop}

\begin{proof}
Since the general fibers of $\overline{f}:Y\rightarrow S$ are in
particular uniruled, it follows that $K_{Y}$ is not $\overline{f}$-pseudo-effective.
By virtue of \cite[Corollary 1.3.3]{BCHM} (see $\S$ \ref{subsec:Minimal-Model-Program}),
there exists a $K_{Y}$-mmp $\theta:Y\dashrightarrow Y'$ relative
to $\overline{f}:Y\rightarrow S$ whose output $\overline{f}':Y'\rightarrow S$
has the structure of a Mori fiber space $\rho':Y'\rightarrow T$ over
some normal $S$-variety $h:T\rightarrow S$. Since for a general
closed point $s\in S$ the restriction $\theta_{s}:Y_{s}\dashrightarrow Y_{s}'$
of $\theta$ is a part of a $K_{Y_{s}}$-mmp ran from the smooth projective
variety $Y_{s}$, it follows from \cite[Corollary 1.7]{HM} that every
irreducible component of pure codimension one of the exceptional locus
of $\theta_{s}$ is uniruled. Since $\theta_{s}$ is a composition
of divisorial contractions and isomorphisms in codimension one, we
deduce from Lemma \ref{lem:Cylinder-birational-pullback} a) and c)
that there exists a sub-cylinder $(Z_{s}',\varphi'_{s})$ of $(Z_{s},\varphi_{s})$
such that $(Z_{s}',\theta_{s}\circ\varphi'_{s})$ is an $\mathbb{A}^{1}$-cylinder
in $Y_{s}'$. Since $Y'$ has terminal singularities and $-K_{Y'}$
is $\rho'$-ample, we deduce from \cite[Corollary 1.4]{HM} that every
fiber of $\rho'$ is rationally chain connected. Since a general closed
fiber of $\rho'$ has again terminal singularities, we deduce in turn
from \cite[Corollary 1.8]{HM} that it is in fact rationally connected.
The assertion then follows from Lemma \ref{lem:Factorization-MFS}. 
\end{proof}
\begin{lem}
\label{lem:MMP-Cylinder} In the setting of Proposition \ref{prop:Relative-MMP},
suppose further that for a general closed point $s\in S$, the $\mathbb{A}^{1}$-cylinder
$(Z_{s},\varphi_{s})$ in $Y_{s}\setminus\Delta_{s}$ is maximally
vertical. Then $Y\setminus\Delta$ is $\mathbb{A}^{1}$-cylindrical
over $S$. 
\end{lem}

\begin{proof}
Since $Y_{s}$ is projective, the hypothesis that $(Z_{s},\varphi_{s})$
is maximally vertical in $Y_{s}\setminus\Delta_{s}$ implies that
the subset $\Delta_{0}$ of irreducible components of $\Delta$ which
are horizontal for $\overline{f}:Y\rightarrow S$ is not empty. Furthermore,
for a general closed point $s\in S$, $\Delta_{0,s}$ intersects the
closures in $Y_{s}$ of the general fibers of $\mathrm{pr}_{Z_{s}}\circ\varphi_{s}^{-1}:\varphi_{s}(Z_{s}\times\mathbb{A}^{1})\rightarrow Z_{s}$
in a unique place. Let $(Z_{s}',\varphi'_{s})$ be a sub-$\mathbb{A}^{1}$-cylinder
of $(Z_{s},\varphi_{s})$ with the property that $(Z_{s}',\theta_{s}\circ\varphi'_{s})$
is an $\mathbb{A}^{1}$-cylinder in $Y_{s}'$ and $\alpha_{s}:Z'_{s}\hookrightarrow T_{s}$
is an open embedding. Since the only divisors that could be contracted
by $\theta_{s}:Y_{s}\dashrightarrow Y'_{s}$ are uniruled hence do
not dominate $Z'_{s}$, we can assume up to shrinking $Z'_{s}$ further
if necessary that the restriction of $\theta_{s}^{-1}$ to ${\rho_{s}'}^{-1}(Z'{}_{s})$
is an isomorphism onto its image $V_{s}$ in $Y_{s}$. Consequently,
$\rho_{s}'\circ\theta_{s}|_{V_{s}}:V_{s}\rightarrow Z'{}_{s}$ is
a $\mathbb{P}^{1}$-fibration extending $\mathrm{pr}_{Z'_{s}}\circ{\varphi'_{s}}^{-1}:\varphi'_{s}(Z'{}_{s}\times\mathbb{A}^{1})\rightarrow Z'{}_{s}$.
Since $(Z_{s},\varphi_{s})$ is vertically maximal in $Y_{s}\setminus\Delta_{s}$,
so is $(Z'_{s},\varphi'_{s})$, and it follows that $\Delta_{0,s}\cap V_{s}$
is a section of $\rho_{s}'\circ\theta_{s}|_{V_{s}}:V_{s}\rightarrow Z'{}_{s}$.
This implies in turn that $\Delta_{0}$ is irreducible and that there
exists an open subset $T_{0}$ of $T$ such that $(\rho'\circ\theta)^{-1}(T_{0})\simeq T_{0}\times\mathbb{P}^{1}$
and $(\rho'\circ\theta)^{-1}(T_{0})\setminus\Delta\simeq T_{0}\times\mathbb{A}^{1}$.
So $Y\setminus\Delta$ is $\mathbb{A}^{1}$-cylindrical over $T_{0}$
whence over $S$. 
\end{proof}

\subsection{General case}

The case of a general morphism $f:X\rightarrow S$ between normal
algebraic varieties is now obtained as follows. By desingularization
theorems \cite{Hi64}, we can find a desingularization $\sigma:\tilde{X}\rightarrow X$
which restricts to an isomorphism over the regular locus $X_{\mathrm{reg}}$.
Since $X$ is normal, it follows in particular that the image of the
exceptional locus of $\sigma$ is a closed subset of $X$ of codimension
at least two. By Nagata completion theorems \cite{Na62} and desingularization
theorems again, there exists an open embedding $j:\tilde{X}\hookrightarrow\tilde{Y}$
into a smooth algebraic variety $\tilde{Y}$ proper over $S$. Then
by Chow lemma \cite[5.6.1]{EGAII} there exists a smooth algebraic
variety $Y$ projective over $S$, say $\overline{f}:Y\rightarrow S$,
and a birational morphism $\tau:Y\rightarrow\tilde{Y}$. Applying
desingularization again, we can further assume that the reduced total
transform of $\tilde{Y}\setminus j(\tilde{X})$ in $Y$ is an SNC
divisor $\Delta$. Since $\tilde{Y}$ is smooth, the image of the
exceptional locus of $\tau$ has codimension at least two in $\tilde{Y}$,
and so the image of the exceptional locus of $\beta=\sigma\circ\tau|_{\tau^{-1}(\tilde{X})}:\tau^{-1}(\tilde{X})\rightarrow X$
is a closed subset of codimension at least two in $X$. Summing up,
we get a sequence of birational maps of $S$-varieties \[\begin{tikzcd} X \arrow[r,dashed, "\sigma^{-1}"] \arrow[d,swap, "f"] \arrow[rrr, bend left=40, "\delta"] & \tilde{X} \arrow[r,hook, "j"] & \tilde{Y} \arrow[r,dashed, "\tau^{-1}"] & Y \arrow[d,"\overline{f}"] \\ S \arrow[rrr,equal] & & & S\end{tikzcd}\]which
we refer to as a \emph{good relative smooth projective completion}
of $f:X\rightarrow S$.
\begin{lem}
\label{lem:good-completion-A1ruled-fibers} Let $f:X\rightarrow S$
be a morphism between normal $k$-varieties and let $\delta:X\dashrightarrow Y$
be a good relative smooth projective completion of $f:X\rightarrow S$.
Suppose that for a general closed point $s\in S$, $X_{s}$ contains
an $\mathbb{A}^{1}$-cylinder $(Z_{s},\varphi_{s})$ over a $\kappa(s)$-variety
$Z_{s}$. Then for a general closed point $s$, there exists a dense
open subset $Z'_{s}$ of $Z_{s}$ such that $(Z'_{s},\delta_{s}\circ\varphi_{s})$
is an $\mathbb{A}^{1}$-cylinder in in $Y_{s}\setminus\Delta_{s}$.
Furthermore, if $(Z_{s},\varphi_{s})$ is vertically maximal in $X_{s}$
then $(Z'_{s},\delta_{s}\circ\varphi_{s})$ is vertically maximal
in $Y_{s}\setminus\Delta_{s}$. 
\end{lem}

\begin{proof}
Since $X$ is normal, for a general closed point $s\in S$, $X_{s}$
is a normal variety. The morphism $(\sigma\circ\tau)_{s}:\tau^{-1}(j(\tilde{X}))_{s}\rightarrow X_{s}$
being proper and birational by construction, the first assertion follows
from Lemma \ref{lem:Cylinder-birational-pullback} b). The second
one is clear from the definition of $\Delta$. 
\end{proof}
The following proposition combined with Proposition \ref{prop:Relative-MMP},
Lemma \ref{lem:MMP-Cylinder} Lemma \ref{lem:good-completion-A1ruled-fibers}
completes the proof of Theorem \ref{thm:MainThm}. 
\begin{prop}
\label{lem:Descending-cylinders-from-completion} Let $f:X\rightarrow S$
be a morphism between normal $k$-varieties and let $\delta:X\dashrightarrow Y$
be a good relative smooth projective completion of $f:X\rightarrow S$.
Suppose that for a general closed point $s\in S$, $X_{s}$ contains
a vertically maximal $\mathbb{A}^{1}$-cylinder $(Z_{s},\varphi_{s})$
over a non uniruled $\kappa(s)$-variety $Z_{s}$. If $Y\setminus\Delta$
is $\mathbb{A}^{1}$-cylindrical over $S$ then so is $X$.
\end{prop}

\begin{proof}
Let $\psi:T\times\mathbb{A}^{1}\hookrightarrow Y\setminus\Delta$
be an $\mathbb{A}^{1}$-cylinder in $Y$ over $S$. It is enough to
show that the restriction of $\mathrm{pr}_{T}$ to the inverse image
by $\psi$ of the exceptional locus $\mathrm{Exc}(\beta)$ of $\beta=\sigma\circ\tau|_{\tau^{-1}(\tilde{X})}:\tau^{-1}(\tilde{X})\rightarrow X$
is not dominant. Indeed, if so, there exists an open subset $T_{0}$
of $T$ such that $\psi(T_{0}\times\mathbb{A}^{1})$ is contained
in $Y\setminus\mathrm{Exc}(\beta)\cup\Delta\simeq\delta(X\setminus\beta(\mathrm{Exc}(\beta)))$.
So suppose on the contrary that there exists an irreducible component
$E$ of $\mathrm{Exc}(\beta)$ such that $\mathrm{pr}_{T}|_{\psi^{-1}(E)}$
is dominant. For a general closed point $s\in S$, the fiber $Y_{s}$
is smooth and the restriction $\beta_{s}:\tau_{s}^{-1}(\tilde{X}_{s})\rightarrow X_{s}$
is an isomorphism outside a closed subset of codimension at least
two in $X_{s}$. So there exists a dense open subset $Z_{s}'$ of
$Z_{s}$ such that $(Z'_{s},\varphi'_{s}=\beta_{s}^{-1}\circ\varphi_{s}|_{Z'_{s}\times\mathbb{A}^{1}})$
is $\mathbb{A}^{1}$-cylinder in $\tau_{s}^{-1}(\tilde{X}_{s})$.
Since $(Z_{s},\varphi_{s})$ is a vertically maximal $\mathbb{A}^{1}$-cylinder
in $X_{s}$, $(Z_{s}',\varphi'_{s})$ is vertically maximal in $\tau_{s}^{-1}(\tilde{X}_{s})$.
On the other hand, for a general closed point $s\in S$, the restriction
$\psi_{s}:T_{s}\times\mathbb{A}^{1}\rightarrow\tau_{s}^{-1}(\tilde{X}_{s})$
is also an embedding. Since $Z_{s}$ whence $Z_{s}'$ is not uniruled,
it follows from Proposition \ref{prop:UniqueCylinder-vertMax} that
$(Z_{s}',\varphi'_{s})$ and $(T_{s},\psi_{s})$ are equivalent $\mathbb{A}^{1}$-cylinders
in $\tau_{s}^{-1}(\tilde{X}_{s})$. But then the restriction of $\mathrm{pr}_{Z_{s}'}$
to ${\varphi_{s}'}^{-1}(E)$ would be dominant, implying in turn that
$(Z_{s}',\varphi_{s}')$ is a not a cylinder, a contradiction. 
\end{proof}
\bibliographystyle{amsplain}

\end{document}